\documentclass[twoside]{article}[11pt]

\usepackage{amssymb,amsmath,amsthm}
\usepackage[colorlinks,linkcolor=violet,citecolor=violet]{hyperref}
\usepackage[nameinlink]{cleveref}
\usepackage{float}
\usepackage{fullpage}
\usepackage[dvipsnames,table]{xcolor}
\usepackage{fancyhdr,datetime}
\usepackage{graphicx}
\usepackage{subcaption}
\usepackage{tikz,pgf,tikz-3dplot,tikz-cd}
\usepackage{etoolbox}
\usepackage{mathrsfs}
\usepackage[percent]{overpic}
\usepackage[export]{adjustbox}
\usepackage{stmaryrd}
\usepackage{stix}

\let\oldsection\section 
\renewcommand{\section}{
  \renewcommand{\theequation}{\thesection.\arabic{equation}}
  \oldsection}

\newcommand{\mexp}[1]{\ensuremath{\exp(-2\pi \mathrm{i}\,\sprod{ #1 })}}

\newcommand{\N}{\mathbb{N}}

\newcommand\Z{\mathbb{Z}}
\newcommand\C{\mathbb{C}}
\newcommand\Q{\mathbb{Q}}
\newcommand\R{\mathbb{R}}

\newcommand{\nops}[1]{\ensuremath{\vert #1  \vert}}

\newcommand{\norm}[1]{\left\lVert#1\right\rVert} 

\newcommand{\trace}{\ensuremath{\mathrm{Trace}}} 
\newcommand{\tbox}[1]{\ensuremath{\quad \mbox{#1} \quad}} 

\newcommand{\RootA}[1][n-1]{\ensuremath{\mathrm{A}_{#1}}}
\newcommand{\RootB}[1][n]{\ensuremath{\mathrm{B}_{#1}}}
\newcommand{\RootC}[1][n]{\ensuremath{\mathrm{C}_{#1}}}
\newcommand{\RootD}[1][n]{\ensuremath{\mathrm{D}_{#1}}}
\newcommand{\RootE}[1][n]{\ensuremath{\mathrm{E}_{#1}}}
\newcommand{\RootF}[1][n]{\ensuremath{\mathrm{F}_{#1}}}
\newcommand{\RootG}[1][n]{\ensuremath{\mathrm{G}_{#1}}}

\newcommand{\weyl}{\ensuremath{\mathcal{W}}} 

\newcommand{\PC}{\ensuremath{\Lambda\!\!\Lambda}}

\newcommand{\roots}{\ensuremath{\rho}}
\newcommand{\highestroot}{\roots_{0}}
\newcommand{\Roots}{\ensuremath{\mathrm{R}}}
\newcommand{\Base}{\ensuremath{\mathrm{B}}}
\newcommand{\Corootlattice}{\ensuremath{\Lambda}}
\newcommand{\sprod}[1]{\ensuremath{\langle #1 \rangle}} 

\newcommand{\fweight}[1]{\ensuremath{\omega_{#1}}}	
\newcommand{\weight}{\ensuremath{\mu}} 
\newcommand{\Weights}{\Omega} 

\newcommand{\Vor}{\ensuremath{\mathrm{Vor}}}

\def\row#1/#2!{#1_{\IfStrEq{#2}{}{n-1}{#2}} & \dynkin{#1}{#2}\\}

\newtheorem{lemma}{Lemma}[section]
\newtheorem{example}[lemma]{Example}

\newtheorem{proposition}[lemma]{Proposition}

\newtheorem{theorem}[lemma]{Theorem}

{\scshape}{\slshape}

\newcommand{\italgf}{\slshape  }

\title{
	On symmetry adapted bases in trigonometric optimization
}

\author{	
	Tobias Metzlaff\thanks{
		Department of Mathematics, University of Kaiserslautern--Landau, 67663 Kaiserslautern, Germany \\
		$\phantom{1111}$\texttt{tobias.metzlaff@rptu.de}}
}

\renewcommand{\headsep}{8mm}
\renewcommand{\footskip}{15mm}

\begin{document}

\maketitle
\thispagestyle{empty}

\begin{abstract}
The problem of computing the global minimum of a trigonometric polynomial is computationally hard. 
We address this problem for the case, where the polynomial is invariant under the exponential action of a finite group. 
The strategy is to follow an established relaxation strategy in order to obtain a converging hierarchy of lower bounds. 
Those bounds are obtained by numerically solving semi-definite programs (SDPs) on the cone of positive semi-definite Hermitian Toeplitz matrices, 
which is outlined in the book of Dumitrescu \cite{dumitrescu07}. 
To exploit the invariance, we show that the group has an induced action on the Toeplitz matrices 
and prove that the feasible region of the SDP can be restricted to the invariant matrices, whilst retaining the same solution. 
Then we construct a symmetry adapted basis tailored to this group action, 
which allows us to block-diagonalize invariant matrices and thus reduce the computational complexity to solve the SDP. 

The approach is in its generality novel for trigonometric optimization and complements the one that was proposed as a poster at the ISSAC 2022 conference \cite{chromaticissac22} and later extended to \cite{chromatic22}. 
In the previous work, we first used the invariance of the trigonometric polynomial to obtain a classical polynomial optimization problem on the orbit space and subsequently relaxed the problem to an SDP. 
Now, we first make the relaxation and then exploit invariance. \\
~\\
Partial results of this article have been presented as a poster at the ISSAC 2023 conference \cite{issac23}. 
\clearpage
\end{abstract}

\tableofcontents

\clearpage

\section{Introduction}
\label{section_introduction}
\setcounter{equation}{0}

We present a sums-of-squares-based algorithm to minimize a multivariate trigonometric polynomial under symmetry assumptions, 
which reduces computational complexity whilst preserving numerical accuracy and convergence rate. 
The algorithm allows exploitation of symmetry with respect to any finite group that acts on a lattice. 

Trigonometric polynomials are good $L^2$-approximations for real-valued, periodic functions \cite{MuntheKaas2012}. 
Minimizing a trigonometric polynomial is a problem that arises, for example, in 
filter design \cite{dumitrescu07}, 
computation of spectral bounds \cite{BdCOV}, 
and optimal power flow \cite{bai08,josz18}. 

Compared to classical polynomials, trigonometric ones are associated to a lattice. 
This makes them relevant for a variety of problems in geometry and information theory, with incidence in physics and chemistry, 
where lattices often provide optimal configurations. 
For example, the hexagonal lattice is classically known to be optimal for 
sampling, packing, covering, and quantization in the plane \cite{conway1988a,kunsch05}, 
but also proved, or conjectured, to be optimal for energy minimization problems \cite{Petrache20,faulhuber23}. 
More recently, the $\RootE[8]$ lattice was proven to give an optimal solution for the sphere packing problem and 
a large class of energy minimization problems in dimension $8$ \cite{Viazovska17,Viazovska22}. 
From an approximation point of view, such lattices describe Gaussian cubature \cite{Xu09,Moody2011}, 
a rare occurence on multidimensional domains, 
and are relevant in graphics and computational geometry, 
see for instance \cite{Choudhary20} and references therein. 

The distinguishing feature of the above lattices is their intrinsic symmetry. 
The latter is given by the action of a finite group on the lattice, 
which induces a linear action on trigonometric polynomials. 
For the goal of trigonometric optimization, 
it is this feature that is emphasized and exploited in an optimization context. 

Computing the global minimum of a polynomial is algorithmically hard and 
can be achieved numerically with an algorithm based on sums-of-squares reinforcements \cite{dumitrescu07,josz18}. 
The convergence rate of this approach for trigonometric polynomials was recently shown to be exponential \cite{bach22}. 

The approach is summarized as follows. 
One restricts to suitable finite subsets 
$\Weights_d \subseteq \Weights_{d+1} \subseteq \ldots \subseteq \Weights$ of the lattice, to obtain a degree bound. 
Then a lower bound for the minimum of a trigonometric polynomial $f$ is obtained 
by computing the maximal $r \in \R$, such that $f - r$ is a Hermitian sum of squares, 
where the summands are bounded in the degree. 
This means, that $f - r$ can be represented as a positive semi-definite Hermitian Toeplitz matrix 
so that the problem of computing $r$ becomes a semi-definite program (SDP). 
By increasing the parameter $d$, that is, the degree bound, one improves the quality of the approximation, 
but the problem becomes more and more costly to solve. 

It is at this point, where we exploit symmetry. 
If the objective function is invariant under the action of a finite group, 
then the Hermitian matrices in the SDP can be block diagonalized according to a symmetry adapted basis. 
This reduces the number of variables of the SDP significantly 
and thus reduces required memory for storing and computational effort for solving the problem. 

The main tool in this article, the symmetry adapted basis, is computed by 
decomposing a representation that corresponds to the finite subset $\Weights_d$ into isotypic components \cite{serre77}. 
The fact that one can do so is widely used in several areas, such as 
combinatorics \cite{Stanley1979,Bergeron09}, 
optimization \cite{Gatermann04,vallentin09,riener2013}, 
approximation \cite{KdK23}, 
interpolation \cite{HubertBazan21,HubertBazan22}, 
dynamical systems \cite{gatermann2000,HubertLabahn2013}, 
polynomial systems \cite{Gatermann1990,FaugereRahmany2009}, 
computational invariant theory \cite{HubertLabahn2016,HubertBazan22b}, 
as well as their fields of application \cite{FasslerStiefel92}. 
It was exploited for the special case of the symmetric group in \cite{KdK23} and is here extended to finite groups acting on lattices. 

An alternative approach to exploiting symmetry in trigonometric optimization was explored by Hubert, Moustrou, Riener and the author in \cite{chromaticissac22,chromatic22} and the thesis \cite{TobiasThesis}. 
The idea there was to rewrite an invariant trigonometric polynomial to a classical one and then use relaxation tools from polynomial optimization \cite{lasserre01,parrilo03}. 
The present paper complements this work in the sense that we first rewrite the problem to a semi-definite program and then exploit symmetry. 
One goal is to provide a framework for a fair comparison between polynomial and Hermitian sums of squares. 
This completes a recently presented poster at the ISSAC 2023 conference \cite{issac23}. 

Next to to symmetry, one can also exploit sparsity in the optimization of complex polynomials at a cost of numerical accuracy and convergence, see \cite{victor22,lakshmi20,corbi23}. 
A further step towards improved computational efficiency would be to exploit both sparsity and symmetry. 


The article is structured as follows: 
In \Cref{section_trigonometric}, we define trigonometric polynomials and the associated concepts of periodicity and degree. 
We recall how to rewrite a trigonometric polynomial in terms of a Hermitian Toeplitz matrix. 
The first novelty is to construct a representation of a finite group that induces the action on trigonometric polynomials up to a fixed degree. 
Under the assumption that the trigonometric polynomial is invariant, 
we show that the associated Hermitian Toeplitz matrix is equivariant with respect to the constructed representation 
and can thus be block diagonalized via a symmetry adapted basis according to the isotypic decomposition of the representation. 

The Hermitian sums-of-squares reinforcement is recalled in \Cref{section_optimization}. 
The feasible region of the arising semidefinite program can be restricted to equivariant Hermitian Toeplitz matrices. 
Thus, it suffices to consider block diagonal matrices. 
We show how to apply symmetry reduction in this context with a clarifying example. 

Finally, in \Cref{section_basis}, we show how to compute the symmetry adapted basis for our setup 
and consider a larger example to discuss possible implementations in practice. 

There is a \textsc{Maple} worksheet
dedicated to the examples and computations in this article, 
which requires the \textsc{Maple} package 
\textsc{GeneralizedChebyshev}\footnote{
\textsc{Maple} software: \href{https://maplesoft.com}{https://maplesoft.com}\\
\phantom{111.}\textsc{GeneralizedChebyshev} package: 
\href{https://github.com/TobiasMetzlaff/GeneralizedChebyshev}{https://github.com/TobiasMetzlaff/GeneralizedChebyshev}\\ 
\phantom{111.}documentation of the package: 
\href{https://tobiasmetzlaff.com/html_guides/GeneralizedChebyshevHelp.html}{https://tobiasmetzlaff.com/html\_guides/GeneralizedChebyshevHelp.html}\\
\phantom{111.}worksheet for this article: 
\href{https://tobiasmetzlaff.com/html_guides/trigonometric_symmetry_reduction.html}{https://tobiasmetzlaff.com/html\_guides/trigonometric\_symmetry\_reduction.html}}. 
Beyond that, the package offers a large variety of functionalities such as an implementation of the irreducible root systems and 
computational aspects of multiplicative invariants. 

\section{Trigonometric polynomials with crystallographic symmetry}
\label{section_trigonometric}
\setcounter{equation}{0}

Let $V$ be a finite-dimensional real vector space with inner product $\sprod{\cdot,\cdot}$. 
Let $\Weights$ be a full-dimensional lattice in $V$ with dual lattice $\Corootlattice := \{ \lambda \in V \,\vert\, \forall\,\weight\in\Weights:\,\sprod{\weight,\lambda} \in \Z \}$. 
A function $f: V \to \R$ which is $\Corootlattice$-periodic and $L^2$-integrable on the periodicity domain has a Fourier expansion 
\begin{equation}\label{eq_trig_poly}
	f(u) = \sum\limits_{\weight\in \Weights} f_\weight \, \mathfrak{e}^\weight(u), 
\end{equation}
with $\mathfrak{e}^\weight(u) := \mexp{\weight,u}$ and coefficients $f_{\weight} = \overline{f_{-\weight}} \in \C$, see \cite{dym85,conway1988a}. 
If all but finitely many coefficients $f_\weight$ are zero, then we call $f$ a \textbf{trigonometric polynomial}. 
The set of all trigonometric polynomials is a vector space with basis $\{\mathfrak{e}^\weight\,\vert\,\weight\in\Weights\}$. 

\subsection{Periodicity domain}

The property ``$\Corootlattice$-periodic'' means that, for all $u\in V$ and $\lambda\in\Corootlattice$, 
we have $f(u+\lambda) = f(u)$. 
In particular, $\Corootlattice$ acts on $V$ as an additive group by translation and 
$f$ is constant on all residue classes $u + \Corootlattice$ of the compact torus $V / \Corootlattice$. 
To understand the periodicity domain of a trigonometric polynomial, we consider the \textbf{Vorono\"{i} cell} of $\Corootlattice$ 
\[
	\Vor(\Corootlattice)
:=	\{ u\in V\,\vert\, \forall\,\lambda\in\Corootlattice:\,\norm{u} \leq \norm{u-\lambda} \},
\]
where $\norm{u}:=\sqrt{\sprod{u,u}}$ is the induced norm. 
This is a compact, convex set and tiles the space by $\Corootlattice$-translation, that is,
\begin{equation}\label{eq_VoronoiTiles}
	V
=	\bigcup\limits_{\lambda\in\Corootlattice} (\Vor(\Corootlattice)+\lambda). 
\end{equation}
The interiors of the cells $\Vor(\Corootlattice)+\lambda$ are disjoint and the intersections of two adjacent cells is an entire face of both of them \cite[Ch. 2, \S 1.2]{conway1988a}. 
The origin $0 \in V$ is contained in the cell $\Vor(\Corootlattice) + 0 = \Vor(\Corootlattice)$ itself. 

We conclude that the Vorono\"i cell $\Vor ( \Corootlattice )$ is the closure of the periodicity domain. 

\subsection{Degree concept}

We now introduce the notion of degree for trigonometric polynomials. 
For $d\in\N$, 
\begin{equation}\label{eq_FiniteWeightSet}
	\Weights_d
:=	\Weights \cap d\,\Vor(\Corootlattice)
\end{equation}
is a finite subset of the lattice, because $\Weights$ is discrete and $\Vor(\Corootlattice)$ is compact. 
We say that a trigonometric polynomial $f\neq 0$ with coefficients $f_\weight$ has \textbf{degree} $\deg(f) = d \in \N$ if $\weight \in \Weights_d$ whenever $\weight \neq 0$ and $d$ is minimal with this property. 
The following statement shows that this is a suitable notion of degree. 

\begin{lemma}
We have 
\begin{enumerate}
	\item $\{0\} =\Weights_0 \subseteq \Weights_1 \subseteq \Weights_2 \subseteq \ldots \subseteq \Weights$, 
	\item $\Weights_d + \Weights_{d'} \subseteq \Weights_{d+d'}$ whenever $d,d'\in\N$, 
	\item $\Weights=\bigcup\limits_{d\in\N} \Weights_d$. 
\end{enumerate}
\end{lemma}
\begin{proof}
The first and third statement are clear. 
Furthermore, 
\[
	\Weights_d + \Weights_{d'}
=	\left(\Weights \cap d\,\Vor(\Corootlattice)\right) + \left( \Weights \cap d'\,\Vor(\Corootlattice) \right)
\subseteq
	\Weights \cap \left(d\,\Vor(\Corootlattice) + d'\,\Vor(\Corootlattice) \right)
=	\Weights \cap (d + d')\, \Vor(\Corootlattice)
=	\Weights_{d+d'},
\]
because $\Weights$ is additively closed and $\Vor(\Corootlattice)$ is convex. 
\end{proof}

\begin{figure}[H]
\begin{center}
\includegraphics[width=0.3\textwidth]{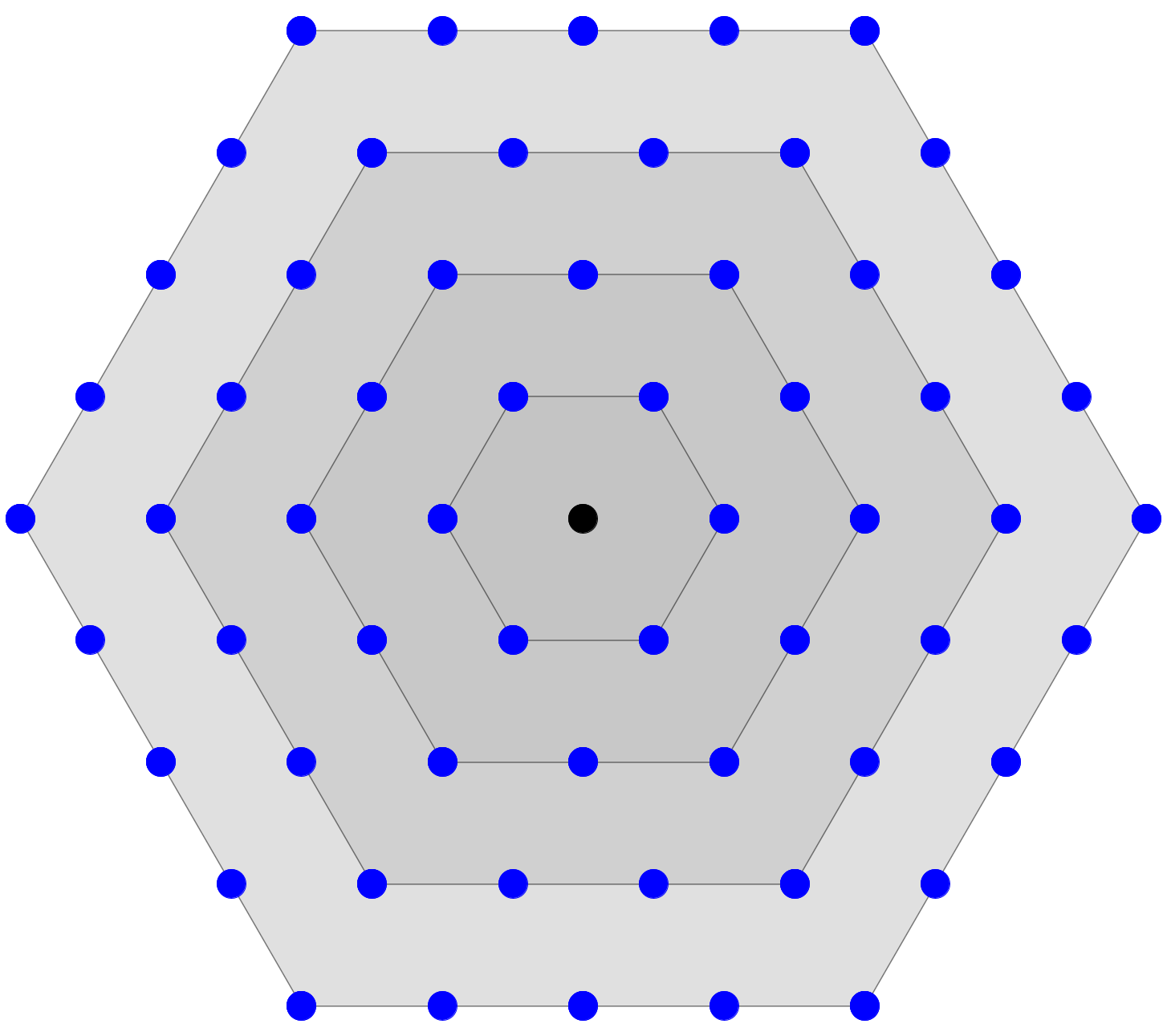}
\caption{The subsets $\{0\} = \Weights_0 \subseteq \Weights_1 \subseteq \ldots \subseteq \Weights_4$ of the hexagonal lattice in the plane 
(specifically the weight lattice of the root system $\RootA[2]$) are contained in scaled copies of the Vorono\"i cell of the dual lattice.}
\label{fig_hexagonal_lattice}
\end{center}
\end{figure}

\begin{example}[Using the notation from \Cref{appendix_root_systems}]~
\begin{enumerate}
\item It may happen that $\Weights_d + \Weights_{d'}$ is a proper subset of $\Weights_{d+d'}$: 
Consider the (self-dual) lattice $\Weights = \Corootlattice = \Z\fweight{1}\oplus\ldots\oplus\Z\fweight{8}$ of the root system $\RootE[8]$ from \emph{\cite[Pl. VII]{bourbaki456}}. 
None of the lattice generators $\fweight{i}$ lies in the Vorono\"{i} cell and thus $\Weights_1 = \Weights_0 = \{0\}$. 
On the other hand, we have $\Weights_2 = \weyl \{0, \fweight{1}, \fweight{8}\}$, where $\weyl$ is the Weyl group. 
Hence, $\Weights_1 + \Weights_1 = \{0\} \subsetneq \Weights_2$. 
\item This definition of degree yields a filtration of the trigonometric polynomials in the sense that the product of two degree $d$ polynomials is of degree at most $2d$. 
However, it is not a grading, as the degree might be strictly smaller than $2d$. 
A counter example is given by the hexagonal lattice in \emph{\Cref{fig_hexagonal_lattice}}. 
\end{enumerate}
\end{example}

\subsection{Toeplitz matrix representations}

Given a trigonometric polynomial $f = \sum_\weight f_\weight\,\mathfrak{e}^\weight$, 
let $d\in\N$ be the minimal integer, such that $\weight\in\Weights_d + \Weights_d$ whenever $f_\weight \neq 0$. 
We shall write $f$ as 
\begin{equation}\label{eq_MatrixRepresentation}
	f(u)
=	\mathbf{E}_d(u)^{\dagger} \,\mathbf{mat}(f)\,\mathbf{E}_d(u),
\end{equation}
where $\dagger$ denotes the complex conjugate transpose and
\begin{itemize}
\item $\mathbf{E}_d(u)$ is the vector of all $\mathfrak{e}^\weight(u)/\sqrt{\nops{\Weights_d}}$ with $\weight\in\Weights_d$, 
\item $\mathbf{mat}(f)$ is a matrix with rows and columns indexed by $\Weights_d$ and independent of $u \in V$. 
\end{itemize}
By comparing coefficients in \Cref{eq_MatrixRepresentation}, we obtain the equation 
\begin{equation}\label{eq_CoeffEntryRelation}
	f_{\eta}
=	\frac{1}{\nops{\Weights_d}}
	\sum\limits_{\substack{\weight,\nu\in\Weights_d \\ \weight - \nu = \eta}} \mathbf{mat}(f)_{\weight \nu} 
\end{equation}
for $\eta \in \Weights$. 
Therefore, we may assume without loss of generality that the entries $\mathbf{mat}(f)_{\weight \nu}$ depend only on the the lattice element $\eta = \weight - \nu$. 
Hence, $\mathbf{mat}(f)$ can be chosen uniquely as a Toeplitz matrix. 
Then the conjugate sign symmetry $f_\weight = \overline{f_{-\weight}}$ implies that $\mathbf{mat}(f)$ is Hermitian, 
that is, $\mathbf{mat}(f)^\dagger = \mathbf{mat}(f)$. 
The space of all Hermitian Toeplitz matrices with rows and columns indexed by $\Weights_d$ is denoted by $\mathrm{Toep}_d$. 

\subsection{Group symmetry}

Let $\weyl\subseteq \mathrm{GL}(V)$ be a finite group. 
We shall define a group action of $\weyl$ on the space of Hermitian Toeplitz matrices. 
In order to do so, we assume that $\Weights$ is a \textbf{$\weyl$-lattice}, that is, $\weyl(\Weights) = \Weights$. 
Then a group element $s\in\weyl$ acts on a trigonometric polynomial $f$ with coefficients $f_\weight$ by 
\begin{equation}
	s \cdot f (u) 
:=	\sum_{\weight \in \Weights} f_\weight\,\mathfrak{e}^{s(\weight)} (u) 
 =	\sum_{\weight \in \Weights} f_{s^{-1}(\weight)}\,\mathfrak{e}^\weight (u)
 =	\sum_{\weight \in \Weights} f_\weight\,\mathfrak{e}^{\weight} ((s^t)^{-1}(u)) ,
\end{equation}
where $s^t$ denotes the transpose of $s$ with respect to $\sprod{\cdot,\cdot}$. 
We call $f$ \textbf{$\weyl$-invariant}, if $s \cdot f = f$ for all $s\in\weyl$. 
To proceed, we need to assume that $\weyl(\Weights_d) = \Weights_d$. 
Then a group element $s\in\weyl$ also acts on a matrix $\mathbf{X} \in \mathrm{Toep}_d$ with entries $\mathbf{X}_{\weight \nu}$ by
\begin{equation}\label{eq_LinearActionMatrix}
	(s \star \mathbf{X})_{\weight \nu} := \mathbf{X}_{s^{-1}(\weight)\,s^{-1}(\nu)} .
\end{equation}
We denote the fixed-point space of this action by 
$\mathrm{Toep}_d^\weyl := \{ \mathbf{X} \in \mathrm{Toep}_d \,\vert\, \forall\,s\in\weyl:\,s \star \mathbf{X} = \mathbf{X} \}$. 

\begin{lemma}
We have $\mathbf{mat}(f)\in\mathrm{Toep}_d^\weyl$ if and only if $f$ is $\weyl$-invariant. 
\end{lemma}
\begin{proof}
Let $s\in\weyl$ and $\eta \in\Weights$. 
Then
\[
	f_{s^{-1}(\eta)}
=	\frac{1}{\nops{\Weights_d}}
	\sum\limits_{\substack{\weight,\nu\in\Weights_d \\ \weight - \nu = s^{-1}(\eta)}} \mathbf{mat}(f)_{\weight \nu} 
=	\frac{1}{\nops{\Weights_d}}
	\sum\limits_{\substack{\weight,\nu\in\Weights_d \\ \weight - \nu = \eta}} \mathbf{mat}(f)_{s^{-1}(\weight) \, s^{-1}(\nu)} 
=	\frac{1}{\nops{\Weights_d}}
	\sum\limits_{\substack{\weight,\nu\in\Weights_d \\ \weight - \nu = \eta}} (s \star \mathbf{mat}(f))_{\weight \, \nu} .
\]
Therefore, we have $s \star \mathbf{mat}(f) = \mathbf{mat}(f)$ if and only if $s \cdot f = f$. 
\end{proof}

\subsection{Symmetry adapted bases}

We show that the matrices $\mathrm{Toep}_d^\weyl$ are block diagonal for a certain choice of basis. 
In order to do so, we start with a representation $\vartheta$ of $\weyl$, which induces the action on $\mathrm{Toep}_d$ by \Cref{eq_LinearActionMatrix}. 

Note that a trigonometric polynomial of degree at most $d$ can be written uniquely in terms of the $\mathfrak{e}^\weight$ with $\weight\in\Weights_d$. 
The coordinates in this basis yield a vector $\mathbf{x} \in \C^{\Weights_d}$ with $\mathbf{x}_{-\weight} = \overline{\mathbf{x}_{\weight}}$, 
where $\C^{\Weights_d}$ denotes the space of complex-valued vectors indexed by $\Weights_d$. 
For $s\in\weyl$, let $\vartheta(s)$ be the permutation matrix that takes $\mathbf{x} \in \C^{\Weights_d}$ to the vector with entries
\begin{equation}
	(\vartheta(s)\,\mathbf{x})_\weight := \mathbf{x}_{s^{-1}(\weight)}. 
\end{equation}
Then we obtain the following version of Maschke's theorem. 

\begin{proposition}\label{prop_SemiSimpleRep}
The space $\C^{\Weights_d}$ is a semi-simple $\weyl$-module with representation $\vartheta$. 
For all $\mathbf{X}\in\mathrm{Toep}_d$, we have
\[
	s \star \mathbf{X}
=	\vartheta(s)\,\mathbf{X}\,\vartheta(s^{-1}). 
\]
\end{proposition}
\begin{proof}
The $\vartheta(s)$ are permutation matrices and thus orthogonal. 
Furthermore, they satisfy $\vartheta(s\,t) = \vartheta(s)\,\vartheta(t)$. 
Hence, $\C^{\Weights_d}$ is a $\weyl$-module with representation $\vartheta$ over $\C$ and as such semi-simple, see \cite[Ch. 1]{serre77}. 
Now, let $s\in\weyl$ and $\mathbf{X}\in\mathrm{Toep}_d$. 
Then
\[
	(\vartheta(s)\,\mathbf{X}\,\vartheta(s^{-1}))_{\weight \nu} 
=	(\vartheta(s)\,\mathbf{X}\,\vartheta(s)^t)_{\weight \nu} 
=	\sum_{\kappa , \eta \in \Weights_d} \vartheta(s)_{\weight \kappa} \, \mathbf{X}_{\kappa \eta} \, \vartheta(s)_{\nu \eta}
=	\mathbf{X}_{s^{-1}\weight \, s^{-1}\nu}
=	(s\star \mathbf{X})_{\weight \nu}  ,
\]
because $\vartheta(s)_{\weight \kappa} = 1$ when $\kappa = s^{-1} \weight$ and $0$ otherwise. 
\end{proof}

In particular, if $\mathbf{X}\in\mathrm{Toep}_d^\weyl$, then $\vartheta(s)\,\mathbf{X} = \mathbf{X}\,\vartheta(s)$. 
Hence, $\mathbf{X}$ commutes with the representation $\vartheta$ and is therefore equivariant. 
Now, since the $\weyl$-module $\C^{\Weights_d}$ is semi-simple, it decomposes into simple submodules. 
By Schur's lemma, any $\weyl$-homomorphism between complex simple modules is either a scalar isomorphism or zero. 
Hence, $\C^{\Weights_d}$ has an \textbf{isotypic decomposition} 
\begin{equation}\label{eq_IsotypicDecomposition}
	\C^{\Weights_d}
	=	\bigoplus\limits_{i=1}^h \left( \bigoplus\limits_{j=1}^{m_i} V_{i\,j} \right),
\end{equation}
where, for all $1 \leq i \leq h$, the $V_{i\,1},\ldots,V_{i\,m_i}$ are isomorphic, simple $\weyl$-submodules with dimension $d_i := \dim(V_{i\,j})$ and multiplicity $m_i\in\N$ so that $\sum_i d_i\,m_i = \nops{\Weights_d}$, see \cite[Ch. 2]{serre77}. 

According to this decomposition, there is a change of basis matrix $\mathbf{T} \in \mathrm{GL}(\C^{\Weights_d})$ 
that transforms any $\mathbf{X} \in \mathrm{Toep}_d^\weyl$ into
\begin{equation}\label{eq_blockstructure}
	\mathbf{T}^{-1}\,\mathbf{X}\,\mathbf{T}
=	\begin{pmatrix}
	\boxed{\mathbf{X}_1}			&			&									\\
									&	\ddots	&									\\
									&			&	\boxed{\mathbf{X}_h}
	\end{pmatrix} 
	\hspace{.3cm} \mbox{with} \hspace{.3cm}
	\mathbf{X}_i
=	\begin{pmatrix}
	\boxed{\tilde{\mathbf{X}}_{i}}	&			&									\\
									&	\ddots	&									\\
									&			&	\boxed{\tilde{\mathbf{X}}_{i}}
\end{pmatrix}
\in	\C^{d_i\,m_i \times d_i\,m_i} , 
\end{equation}
where each $\mathbf{X}_i$ consists of $d_i$ equal blocks $\tilde{\mathbf{X}}_{i}$ of size $m_i \times m_i$. 
By applying Gram-Schmidt, we may assume additionally that $\mathbf{T}$ is unitary, that is, $\mathbf{T}^{-1} = \mathbf{T}^\dagger$. 
The column vectors of $\mathbf{T}$ form a \textbf{symmetry adapted basis} of $\C^{\Weights_d}$. 
In particular, we obtain a basis for the trigonometric polynomials of degree at most $d$ according to the decomposition. 

\section{Optimization}
\label{section_optimization}
\setcounter{equation}{0}

We present a numerical algorithm to minimize a $\weyl$-invariant trigonometric polynomial $f$ with coefficients 
$f_{\weight} = \overline{f_{-\weight}} = f_{s^{-1}(\weight)}$ for $s\in\weyl$ a finite group and $\weight \in\Weights$ a full-dimensional $\weyl$-lattice in a finite-dimensional real vector space $V$ as in the previous \Cref{section_trigonometric}. 
The global minimum 
\[
	f^*
:=	\min\limits_{u\in V} f(u)
\]
exists and is assumed in some minimizer in the compact periodicity domain $\Vor(\Corootlattice)$, where $\Corootlattice$ is the dual lattice of $\Weights$. 

Our strategy is to follow a known relaxation technique of this optimization problem to a hierarchy of semi-definite programs (SDP), whose solutions will provide lower bounds of increasing quality for the minimum. 
This technique is explained in the book of Dumitrescu for the case $\Weights=\Z^n$ \cite[Ch. 2, 3, 4]{dumitrescu07}, which is adapted to any lattice $\Weights$ with the notion of degree ($=$ order of the hierarchy) as it was defined in the previous section in a straightforward manner. 
We then exploit the symmetry of $f$ by writing the SDP matrices in block diagonal form according to the isotypic decomposition in \Cref{eq_blockstructure}. 

The fundamental difference to \cite{chromatic22} is that we first relax to an SDP and then exploit symmetry. 

\subsection{Sums of squares}

To begin, let $f$ be a trigonometric polynomial and let $d\in\N$ be the minimal integer 
so that $\weight\in\Weights_d + \Weights_d$ whenever $f_\weight\neq 0$. 
The number $d$ can be seen as the starting index of a hierarchy of lower bounds, which is non-decreasing and converges to $f^*$. 
We shall see that exploiting symmetry in this setup does not effect the quality of the bound, but reduces the computational effort. 

Recall from \Cref{eq_MatrixRepresentation} that there is a unique Hermitian Toeplitz matrix $\mathbf{mat}(f)\in\mathrm{Toep}_d$, such that
\[
f(u)
=	\mathbf{E}_d(u)^\dagger \, \mathbf{mat}(f) \, \mathbf{E}_d(u).
\]
On the space of Hermitian Toeplitz matrices, the trace admits an inner product, so that we have an SDP
\begin{equation}\label{eq_MinRelaxation}
		f^*
\geq	f_d
:=		\min\limits_{\mathbf{X} \in \mathrm{Toep}_d} \, \trace(\mathbf{mat}(f) \, \mathbf{X}) \quad \mathrm{s.t.} \quad \mathbf{X} \succeq 0 ,\, \trace(\mathbf{X}) = 1 . 
\end{equation}
Here, $\mathbf{X} \succeq 0$ means that $\mathbf{X}$ is Hermitian positive semi-definite. 
This defines an ascending sequence of lower bounds $f_d \leq f_{d+1} \leq \ldots \leq f^*$ with asymptotic convergence, that is, $f^* = f_d$ for $d\to\infty$. 
The convergence rate is exponential \cite{bach22}. 

\begin{example}\label{ex_SOScertificate}
Consider the case $n=1$ with $\Weights = \Corootlattice = \Z$ and $\weyl = \{\pm 1\}$ (see also \emph{\Cref{example_2RootSys}}). 
Then $\Weights_1 = \{1,0,-1\}$ and $\mathbf{E}_1 = (\mathfrak{e}^1,1,\mathfrak{e}^{-1})^t/\sqrt{3}$.
We have
\[
	\mathbf{E}_1^\dagger\,\mathbf{E}_1
=	\frac{1}{3} \, \begin{pmatrix} \mathfrak{e}^{-1} & 1 & \mathfrak{e}^1 \end{pmatrix} \, \begin{pmatrix} \mathfrak{e}^1 \\ 1 \\ \mathfrak{e}^{-1} \end{pmatrix}
=	1
\quad \mbox{and} \quad	
	\mathbf{E}_1 \, \mathbf{E}_1^\dagger
=	\frac{1}{3}
\begin{pmatrix}
	1		&	\mathfrak{e}^1		&	\mathfrak{e}^{2}	\\
	\mathfrak{e}^{-1}	&	1		&	\mathfrak{e}^1		\\
	\mathfrak{e}^{-2}	&	\mathfrak{e}^{-1}	&	1
\end{pmatrix} .
\]
As a toy example, take 
$f(u) := \mathfrak{e}^2(u) - 2 \, \mathfrak{e}^1(u) + 3 - 2 \, \mathfrak{e}^{-1}(u) + \mathfrak{e}^{-2}(u) = 2\,\cos(4\pi u) - 4\,\cos(2\pi u) + 3$ 
with degree $2$ and global minimum $f^* = f(\lambda \pm 1/6) = 0$ for $\lambda \in \Z$. 
Note that $f$ can be rewritten as
\[
	f(u)
=	{\mathbf{E}_1(u)}^\dagger \, \mathbf{mat}(f) \, \mathbf{E}_1(u) ,
	\quad \mbox{where} \quad
	\mathbf{mat}(f)
=	\begin{pmatrix}
	3	&	-3	&	3	\\
	-3	&	3	&	-3	\\
	3	&	-3	&	3
\end{pmatrix} 
\in	\mathrm{Toep}_1 .
\]
Thus, the SDP from \emph{\Cref{eq_MinRelaxation}} is
\[
	f_1
=	\min\limits_{b,c\in\C} \, \trace\left(	
	\begin{pmatrix}
	3	&	-3	&	 3	\\
	-3	&	 3	&	-3	\\
	3	&	-3	&	 3
	\end{pmatrix} \,
	\begin{pmatrix}
	1/3				&	b				& c \\
	\overline{b}	&	1/3				& b \\
	\overline{c}	&	\overline{b}	& 1/3
	\end{pmatrix}
\right)	 
\quad \mathrm{s.t.} \quad 
	\begin{pmatrix}
	1/3						&	b				& c \\
	\overline{b}	&	1/3						& b \\
	\overline{c}	&	\overline{b}	& 1/3
	\end{pmatrix}
\succeq 0
\]
and the optimal value $f^* = f_1 = 0$ is obtained with $b=-c=1/6$.
\end{example}

\subsection{Symmetry reduction}

Now that we have formulated the hierarchy of SDPs, let us assume that $f$ is $\weyl$-invariant. 

\begin{proposition}\label{prop_SDPInvariance}
We have $f_d = f_{d,\weyl}$, where $f_d$ is as in \emph{\Cref{eq_MinRelaxation}} and
\[
	f_{d,\weyl} 
:=	\min\limits_{\mathbf{X} \in \mathrm{Toep}_d^\weyl} \, \trace(\mathbf{mat}(f) \, \mathbf{X}) 
	\quad \mathrm{s.t.} \quad \mathbf{X} \succeq 0 ,\, \trace(\mathbf{X}) = 1.
\]
\end{proposition}
\begin{proof}
If $\mathbf{X}\in\mathrm{Toep}_d^\weyl$ is feasible for $f_{d,\weyl}$, then it is especially feasible for $f_d$. 
Hence, we have $f_d \leq f_{d,\weyl}$. 

For the converse, let $\mathbf{X}\in\mathrm{Toep}_d$ be feasible for $f_{d}$ and define
\[
	\tilde{\mathbf{X}}
:=	\frac{1}{\nops{\weyl}} \sum\limits_{s\in\weyl} s\star \mathbf{X}
\in	\mathrm{Toep}_d^\weyl.
\]
We claim that $\tilde{\mathbf{X}}$ is feasible for $f_{d,\weyl}$. 
Indeed, for $\mathbf{X} \succeq 0$, $s\in\weyl$ and $\mathbf{x} \in \C^{\Weights_d}$, we have 
\[
	{\mathbf{x}}^\dagger \, (s \star \mathbf{X}) \, \mathbf{x} 
=	{\mathbf{x}}^\dagger \, (\vartheta(s) \, \mathbf{X} \, \vartheta(s)^t) \, \mathbf{x} 
=	(\vartheta(s)^t\,\mathbf{x})^\dagger \, \mathbf{X} \, (\vartheta(s)^t \mathbf{x}) 
\geq 0 ,
\]
because $\vartheta(s)^t = \vartheta(s)^\dagger$. 
Thus, $s \star \mathbf{X}$ is Hermitian positive semi-definite. 
Furthermore, if $\trace(\mathbf{X}) = 1$, then
\[
	\trace(s\star \mathbf{X}) 
=	\sum\limits_{\weight\in\Weights_d} \mathbf{X}_{s^{-1}(\weight) \, s^{-1}(\weight)} 
=	\sum\limits_{\nu\in\Weights_d} \mathbf{X}_{\nu \nu} 
=	\trace(\mathbf{X}) 
=	1 . 
\]
Thus, $s \star \mathbf{X}$ is also feasible for $f_d$. 
Since the feasibility region of $f_d$ is convex, it also contains the convex combination $\tilde{\mathbf{X}}$. 
In particular, $\tilde{\mathbf{X}}$ is feasible for $f_{d,\weyl}$ with
\[
	\trace(\mathbf{mat}(f) \, \tilde{\mathbf{X}}) 
=	\frac{1}{\nops{\weyl}} \sum\limits_{s\in\weyl} \trace(\mathbf{mat}(f) \, (s\star \mathbf{X})) 
=	\frac{1}{\nops{\weyl}} \sum\limits_{s\in\weyl} \trace((s^{-1} \star \mathbf{mat}(f)) \, \mathbf{X}) 
=	\trace(\mathbf{mat}(f) \, \mathbf{X}) , 
\]
because $\mathbf{mat}(f) \in \mathrm{Toep}_d^\weyl$. 
Hence, we also have $f_d \geq f_{d,\weyl}$. 
\end{proof}

Recall from \Cref{eq_blockstructure} that matrices in $\mathrm{Toep}_d^\weyl$ can be block diagonalized with a change of basis $\mathbf{T}$. 

\begin{theorem}\label{thm_SDPBlock}
Assume that $\mathbf{T}^{\dagger}\,\mathbf{mat}(f)\,\mathbf{T}$ has blocks $\tilde{\mathbf{F}}_i\in\C^{m_i\times m_i}$. 
Then $f_d = f_{d,\weyl} = f_{d,\weyl}^{\mathrm{block}}$, where 
\[
	f_{d,\weyl}^{\mathrm{block}}
:=	\min\limits_{\mathbf{X} \in \mathrm{Toep}_d^\weyl} \, \sum\limits_{i=1}^h d_i\,\trace(\tilde{\mathbf{F}}_i \, \tilde{\mathbf{X}}_i) 
	\quad \mathrm{s.t.} \quad 
\trace(\mathbf{X}) = 1 ,\, \mathbf{T}^{\dagger}\,\mathbf{X}\,\mathbf{T} \mbox{\emph{ has blocks }} \tilde{\mathbf{X}}_i \succeq 0.
\]
\end{theorem}
\begin{proof}
Since the trace does not depend on the basis, we have 
\[
	\trace(\mathbf{mat}(f) \, \mathbf{X})
=	\trace(\mathbf{T}^{\dagger}\,\mathbf{mat}(f)\,\mathbf{T}\,\mathbf{T}^{\dagger} \, \mathbf{X} \, \mathbf{T}) 
=	\sum\limits_{i=1}^h d_i\,\trace(\tilde{\mathbf{F}}_i \, \tilde{\mathbf{X}}_i) .
\]
Note that $\mathbf{X}\succeq 0$ if and only if $\mathbf{T}^\dagger \, \mathbf{X} \, \mathbf{T}\succeq 0$, that is, for all $1\leq i\leq h$, we have $\tilde{\mathbf{X}}_i \succeq 0$. 
Thus, $f_{d,\weyl} = f_{d,\weyl}^{\mathrm{block}}$ and the statement follows from \Cref{prop_SDPInvariance}. 
\end{proof}

\begin{example}
We continue with \emph{\Cref{ex_SOScertificate}} with $n=1$, $d=1$ and $\weyl=\{\pm \mathrm{id}\}$. 
An arbitrary $\mathbf{X} \in \mathrm{Toep}_1$ shall have rows and columns indexed by $\Weights_1 = \{1,0,-1\}$. 
Then the group element $-\mathrm{id}\in\weyl$ acts on $\mathrm{Toep}_1$ by
\[
\begin{matrix} & 
	\begin{matrix} \phantom{-} \textcolor{gray}{1} & \phantom{-} \textcolor{gray}{0} & \textcolor{gray}{-1} \end{matrix} \\	
	\begin{matrix} \phantom{-} \textcolor{gray}{1} \\ \phantom{-} \textcolor{gray}{0} \\ \textcolor{gray}{-1} \end{matrix} & 
	\begin{pmatrix}
		\phantom{-} a & \phantom{-} b & \phantom{-} c \\
		\phantom{-} \overline{b} & \phantom{-} a & \phantom{-} b \\
		\phantom{-} \overline{c} & \phantom{-} \overline{b} & \phantom{-} a
	\end{pmatrix}
\end{matrix}
\phantom{-} \underrightarrow{\phantom{-} (-\mathrm{id}) \star \ldots \phantom{-}} \phantom{-}
\begin{matrix} & 
	\begin{matrix} \phantom{-} \textcolor{gray}{1} & \phantom{-} \textcolor{gray}{0} & \textcolor{gray}{-1} \end{matrix} \\	
	\begin{matrix} \phantom{-} \textcolor{gray}{1} \\ \phantom{-} \textcolor{gray}{0} \\ \textcolor{gray}{-1} \end{matrix} & 
	\begin{pmatrix}
	\phantom{-} a & \phantom{-} \overline{b} & \phantom{-} \overline{c} \\
		\phantom{-} b & \phantom{-} a & \phantom{-} \overline{b} \\
		\phantom{-} c & \phantom{-} b & \phantom{-} a
	\end{pmatrix}
\end{matrix} 
\]
for $a\in\R$ and $b,c\in\C$. 
Thus, the fixed point space $\mathrm{Toep}_1^\weyl$ consists of those $\mathbf{X} \in \mathrm{Toep}_1$ with $a,b,c \in \R$. 
The above action is induced by the orthogonal representation $\rho : \weyl \to \mathrm{GL}(\C^{\Weights_1}) = \mathrm{GL}_3(\C)$ with
\[
	\rho(-\mathrm{id})
=	\begin{pmatrix}0&0&1\\0&1&0\\1&0&0\end{pmatrix} .
\]
We have an isotypic decomposition
\[
		\C^3
=		\left(\sprod{\begin{pmatrix} 0\\1\\0 \end{pmatrix}}
\oplus	\sprod{\frac{1}{\sqrt{2}}\begin{pmatrix} 1\\0\\1 \end{pmatrix}} \right)
\oplus	\sprod{\frac{1}{\sqrt{2}}\begin{pmatrix} 1\\0\\-1 \end{pmatrix}}
\]
with $h=2$ irreducible representations with multiplicities $m_1=2, m_2=1$ and dimensions $d_1=1, d_2=1$. 
When $\mathbf{T}$ is the matrix with columns given by the above three basis vectors and 
$f$ is the $\weyl$-invariant trigonometric polynomial from \emph{\Cref{ex_SOScertificate}}, then
\[
	\mathbf{T}^\dagger\,\mathbf{X}\,\mathbf{T}
=	\begin{pmatrix}
	a			&	\sqrt{2}\,\overline{b}	&	0	\\
	\sqrt{2}\,b	&	a+c					&	0	\\
	0			&	0						&	a-c
\end{pmatrix} 
\quad \mbox{and} \quad 
\mathbf{T}^\dagger\,\mathbf{mat}(f)\,\mathbf{T}
=	\begin{pmatrix} 3 & -3\,\sqrt{2} & 0 \\ -3\,\sqrt{2} & 6 & 0 \\ 0 & 0 & 0 \end{pmatrix} . 
\]
Note that the condition $\trace(\mathbf{X}) = 1$ implies $a=1/3$. 
Thus,
\[
	f_{1,\weyl}^{\mathrm{block}}
=	\min\limits_{b,c\in\R} \, \trace\left( \begin{pmatrix} 3 & -3\,\sqrt{2} \\ -3\,\sqrt{2} & 6 \end{pmatrix} \, 
	\begin{pmatrix} 1/3 & \sqrt{2}\,b \\ \sqrt{2}\,b & 1/3+c \end{pmatrix} \right) 
\quad \mathrm{s.t.} \quad \begin{pmatrix} 1/3 & \sqrt{2}\,b \\ \sqrt{2}\,b & 1/3+c \end{pmatrix} \succeq 0 ,\, 1/3 \geq c .
\]
Finally, the optimal value $f^* = f_{1} = f_{1,\weyl} = f_{1,\weyl}^{\mathrm{block}} = 0$ is recovered with $b=-c=1/6$.
\end{example}

\section{Computing symmetry adapted bases}
\label{section_basis}
\setcounter{equation}{0}

We have seen the advantage of using a symmetry adapted basis in trigonometric optimization in the previous \Cref{section_optimization}. 
Using the computer algebra systems 
\textsc{Oscar}\footnote{\href{https://www.oscar-system.org}{https://www.oscar-system.org}} and 
\textsc{Maple}\footnote{\href{https://maplesoft.com}{https://maplesoft.com}}, 
we explain how to obtain said basis in practice with the algorithm from \cite[Ch. 2]{serre77}. 

\subsection{Projection onto isotypic components}

The semi-simple $\weyl$-module $\C^{\Weights_d}$ has an isotypic decomposition
\[
	\C^{\Weights_d}
=	V^{(1)} \oplus \ldots \oplus V^{(h)},
	\quad
	V^{(i)}
=	V^{(i)}_1 \oplus \ldots \oplus V^{(i)}_{m_i},
\]
where, for each $i$, the $V^{(i)}_j$ are isomorphic, simple $\weyl$-submodules with multiplicity $m_i$ and dimensions $d_i = \dim(V^{(i)}_j)$. 
The characters $\chi_i$ of the representations associated to the $V^{(i)}_j$ are the complex irreducible characters of $\weyl$. 

\begin{proposition}
\label{prop_projection}
\emph{\cite{serre77}}
For $1\leq i\leq h$, a projection of $V$ onto $V^{(i)}$ is given by
\[
	\mathbf{P}^{(i)}
:=	\frac{d_i}{\nops{\weyl}} \sum\limits_{s\in\weyl} \chi_i(s^{-1}) \, \vartheta(s) .
\]
\end{proposition}

\subsection{Computing the basis}

To compute a basis for $\C^{\Weights_d}$ so that the $\weyl$--invariant matrices have the block diagonal structure from \Cref{eq_blockstructure}, 
we fix an index $1 \leq i \leq h$ and follow the steps below. 

\begin{enumerate}
\item For an irreducible representation of $\weyl$ with character $\chi_i$, choose representing matrices $(\mathbf{D}^{(i)}_{k\,\ell}(s))_{1 \leq k , \ell \leq d_i}$. 

\item For $1\leq \ell\leq d_i$, set 
\[
	\mathbf{P}^{(i)}_\ell
:=	\frac{d_i}{\nops{\weyl}} \sum_{s\in\weyl} \mathbf{D}^{(i)}_{1\,\ell} (s^{-1}) \, \vartheta(s). 
\]
The column dimension of this matrix is $m_i$. 

\item Compute an orthonormal basis $\left\{w^{(i)}_{1\,1},\ldots,w^{(i)}_{1\,m_i}\right\}$ for the column space of $\mathbf{P}^{(i)}_1$. 

\item For $2\leq \ell\leq d_i$ and $1\leq j \leq m_i$, set
\[
	w^{(i)}_{\ell\,j}
:=	\mathbf{P}^{(i)}_\ell \, w^{(i)}_{1\,j}.
\]
\end{enumerate}

Note that the representing matrices in step $1.$ are by no means unique. 
Finding $(\mathbf{D}^{(i)}_{k\,\ell}(s))_{1 \leq k , \ell \leq d_i}$ is a problem on its own and addressed in the example below. 

\begin{proposition}
\label{prop_ChangeOfBasis}
\emph{\cite{serre77,FasslerStiefel92}}
For $1\leq i\leq h$, $1\leq j\leq m_i$, the set $\left\{w^{(i)}_{1 \, j} , \ldots , w^{(i)}_{d_i \, j}\right\}$ is an orthonormal basis for $V^{(i)}_{j}$. 

Furthermore, the matrix
\[
	\tilde{\mathbf{T}}
:=	\bigg(
	\underbrace{w^{(1)}_{1 \, 1} , \ldots , w^{(1)}_{d_1 \, 1} }_{V^{(1)}_{1}}
	,
	\ldots
	,
	\underbrace{w^{(1)}_{1 \, m_1} , \ldots , w^{(1)}_{d_1 \, m_1} }_{V^{(1)}_{m_1}}
	,
	\ldots\ldots
	,
	\underbrace{w^{(h)}_{1 \, 1} , \ldots , w^{(h)}_{d_h \, 1} }_{V^{(h)}_{1}}
	,
	\ldots
	,
	\underbrace{w^{(h)}_{1 \, m_h} , \ldots , w^{(h)}_{d_h \, m_h} }_{V^{(h)}_{m_h}}
	\bigg)
\]
has rank $\nops{\Weights_d}$ and, for $s\in\weyl$, $\tilde{\mathbf{T}}^{\dagger}\,\vartheta(s)\,\tilde{\mathbf{T}}$ has $h$ blocks, each consisting of $m_i$ further equal $d_i\times d_i$ blocks. 

On the other hand, when the columns are reordered to 
\[
	\mathbf{T}
:=	\bigg(
	w^{(1)}_{1 \, 1} , \ldots , w^{(1)}_{1 \, m_1} 
	,
	\ldots
	,
	w^{(h)}_{d_1 \, 1} , \ldots , w^{(h)}_{d_1 \, m_1} 
	,
	\ldots\ldots
	,
	w^{(h)}_{1 \, 1} , \ldots , w^{(h)}_{1 \, m_h} 
	,
	\ldots
	,
	w^{(h)}_{d_h \, 1} , \ldots , w^{(h)}_{d_h \, m_h} 
	\bigg),
\]
then, for $\mathbf{X} \in \mathrm{Toep}_d^\weyl$, $\mathbf{T}^{\dagger}\,\mathbf{X}\,\mathbf{T}$ has $h$ blocks, each consisting of $d_i$ further equal $m_i\times m_i$ blocks (see \emph{\Cref{eq_blockstructure}}). 
\end{proposition}

\subsection{Example}

We consider the Weyl group of the $\RootA[2]$ root system $\weyl \cong \mathfrak{S}_3$. 
The group has order $\nops{\weyl} = 6$ and $ h = 3$ irreducible representations, which is encoded by the character table $\chi$. 

\subsubsection{Computing multiplicities}

We think of $\chi \in \mathrm{GL}_h(\Z)$ as an invertable integer matrix with rows indexed by the irreducible representations $1\leq i\leq h$ 
and columns indexed by the conjugacy classes of the group elements $s\in\weyl$,
such that
\[
\trace(\vartheta(s)) = m_1\,\chi_{1}(s) + \ldots + m_h\,\chi_{h}(s) .
\]
In particular, the multiplicities $m_i$ are the solutions of a linear system and the dimensions are $d_i=\chi_{i}(\mathrm{id})$, where $\mathrm{id} \in \weyl$ is the identity element. 
In the case of $\weyl \cong \mathfrak{S}_3$, we have
\[
	\chi
=	\begin{matrix}
	& \begin{matrix}  \phantom{-} \textcolor{gray}{\mathrm{id}} & \phantom{..} \textcolor{gray}{s_1} \phantom{..} & \textcolor{gray}{s_1\,s_2} \end{matrix} \\
	\begin{matrix} \phantom{-} \textcolor{gray}{1} \\ \phantom{-} \textcolor{gray}{2} \\ \phantom{-} \textcolor{gray}{3} \end{matrix} & 
	\begin{pmatrix} \phantom{-} 1 \phantom{.}	&			 -  1 \phantom{-}	&  \phantom{-} 1\\ 
					\phantom{-} 2 \phantom{.}	&	\phantom{-} 0 \phantom{-}	&   		-  1	&\\ 
					\phantom{-} 1 \phantom{.}	&	\phantom{-} 1 \phantom{-}	&  \phantom{-} 1 	&\end{pmatrix}
	\end{matrix} , 
\] 
where $s_1 = (1\,2),s_2 = (2\,3) \in \weyl$ are the generating transpositions of $\mathfrak{S}_3$. 
With the computer algebra system \textsc{Oscar}\footnote{\href{https://www.oscar-system.org}{https://www.oscar-system.org}}, 
this matrix can be computed as follows. 
\begin{verbatim}
julia> using Oscar;
julia> W = symmetric_group(3);
julia> X = character_table(W);
\end{verbatim}
Note that the first and third row of $\chi$ correspond to the sign and trivial representation of dimension $1$, respectively. 
By computing $\trace(\vartheta(s))$ and solving for the $m_i$, we find the multiplicities in \Cref{table_A2_multiplicities}. 
\begin{table}[H]
	\begin{center}
		\begin{tabular}{|c||c|c|c|c|c|c|}
			\hline
					&	$d=1$	&	$d=2$	&	$d=3$	&	$d=4$	&	$d=5$	&	$d=6$	\\
			\hline
			\hline
			$m_1$	&	$ 0$	&	$ 1$	&	$ 3$	&	$ 6$	&	$10$	&	$15$	\\
			\hline
			$m_2$	&	$ 2$	&	$ 6$	&	$12$	&	$20$	&	$30$	&	$42$	\\
			\hline
			$m_3$	&	$ 3$	&	$ 6$	&	$10$	&	$15$	&	$21$	&	$28$	\\
			\hline
		\end{tabular}
	\end{center}
\caption{The multiplicities $m_i$ of the irreducible representations of $\weyl \cong \mathfrak{S}_3$ 
	occurring in the representation $\vartheta$ by which $\weyl$ acts on the trigonometric polynomials up to degree $d$. 
	Those are the $m_1$: sign-, $m_2$: reflection- and $m_3$: trivial representation.}
\label{table_A2_multiplicities}
\end{table}
For $d=1$, we have $\Weights_1 = \{0\} \cup \weyl\fweight{2} \cup \weyl\fweight{1} = \{0,-\fweight{1},\fweight{1}-\fweight{2},\fweight{2},-\fweight{2},\fweight{2}-\fweight{1},\fweight{1}\}$ and, 
by \Cref{prop_projection}, the projections of $\C^{\Weights_1} = \C^7$ onto $V_1, V_2, V_3$ are 
\[
	\mathbf{P}^{(1)}
=	0 ,
	\quad
	\mathbf{P}^{(2)}
=	\frac{1}{3} \begin{pmatrix}
	 0 &  0 &  0 &  0 &  0 &  0 &  0 \\
	 0 &  2 & -1 & -1 &  0 &  0 &  0 \\
	 0 & -1 &  2 & -1 &  0 &  0 &  0 \\
	 0 & -1 & -1 &  2 &  0 &  0 &  0 \\
	 0 &  0 &  0 &  0 &  2 & -1 & -1 \\
	 0 &  0 &  0 &  0 & -1 &  2 & -1 \\
	 0 &  0 &  0 &  0 & -1 & -1 &  2 	
\end{pmatrix} ,
	\quad
	\mathbf{P}^{(3)}
=	\frac{1}{3} \begin{pmatrix} 
	3 & 0 & 0 & 0 & 0 & 0 & 0 \\
	0 & 1 & 1 & 1 & 0 & 0 & 0 \\
	0 & 1 & 1 & 1 & 0 & 0 & 0 \\
	0 & 1 & 1 & 1 & 0 & 0 & 0 \\
	0 & 0 & 0 & 0 & 1 & 1 & 1 \\
	0 & 0 & 0 & 0 & 1 & 1 & 1 \\
	0 & 0 & 0 & 0 & 1 & 1 & 1 	
\end{pmatrix} .
\]

\subsubsection{Computing representing matrices}

Next, we need the representing matrices associated to the irreducible characters. 
The representations associated to the characters $\chi_1$ and $\chi_2$ are $1$-dimensional. 
Hence, the representing matrices are 
\[
	\mathbf{D}^{(1)}(s_1)
=	\mathbf{D}^{(1)}(s_2)
=	(-1)
	\tbox{and}
	\mathbf{D}^{(3)}(s_1)
=	\mathbf{D}^{(3)}(s_2)
=	(1) .
\]
For $\chi_2$, we discuss three possible approaches. 

\subsubsection*{Over $\Z$ with \textsc{Maple}}

This strategy works, because $\weyl$ is a Weyl group (see \Cref{appendix_root_systems}): 
It acts by permutation of coordinates on $V = \R^3/\sprod{[1,1,1]^t}$, 
which is a simple $\weyl$-module with character $\chi_2$. 
The root system $\RootA[2]$ from \Cref{example_rootsystemA2} is a root system for $V$ and 
the weight lattice $\Weights = \Z \, \fweight{1} \oplus \Z \, \fweight{2}$ is hexagonal. 
The reflection associated to the root $\roots_i$ takes $\fweight{j}$ to
\[	
	\fweight{j} - \sprod{\fweight{j},\roots_i^\vee} \, \roots_i
=	\begin{cases}
		\fweight{2} - \fweight{1},	&	\mbox{if } i=j=1	\\
		\fweight{2},				&	\mbox{if } i=1,\,j=2\\
		\fweight{1},				&	\mbox{if } i=2,\,j=1\\
		\fweight{1} - \fweight{2},	&	\mbox{if } i=j=2
	\end{cases} .
\]
The row vectors of the representing matrices are the coordinates in the basis of the $\fweight{i}$, that is,
\[
	\mathbf{D}^{(2)}(s_1)
=	\begin{pmatrix}
	-1 & 1 \\ 0 & 1
	\end{pmatrix}
	\tbox{and}
	\mathbf{D}^{(2)}(s_2)
=	\begin{pmatrix}
	1 & 0 \\ 1 & -1
	\end{pmatrix} .
\]
To obtain these matrices, one can use the following command from the \textsc{Maple} package 
\textsc{GeneralizedChebyshev}: 
\begin{verbatim}
maple> with(LinearAlgebra): 
maple> read("GeneralizedChebyshev.mpl"): with(GeneralizedChebyshev):
maple> ZWeylGroupGen(A,2);
\end{verbatim}

\subsubsection*{Over $\Z[\zeta]$ with \textsc{Oscar}}

Since $\weyl$ is a reflection group, it can be realized over $\Q[\zeta]$, 
where $\zeta \in \C$ is a third root of unity \cite{shephardtodd54}. 
The matrices of $s_1$ and $s_2$ in this reflection representation are
\[
	\mathbf{D}^{(2)}(s_1)
=	\begin{pmatrix}
	0 & -\zeta-1 \\ \zeta & 0
\end{pmatrix}
\tbox{and}
\mathbf{D}^{(2)}(s_2)
=	\begin{pmatrix}
	0 & 1 \\ 1 & 0
\end{pmatrix} .
\]
Note that the coefficients are actually in $\Z$. 
With \textsc{Oscar}, those are obtained as follows.
\begin{verbatim}
julia> using Oscar;
julia> W = symmetric_group(3);
julia> V = irreducible_modules(W);
julia> g1, g2 = V[3].ac[1].matrix, V[3].ac[2].matrix;
julia> s1, s2 = g2, g1*g2
\end{verbatim}
(In the \textsc{Oscar} version used for this article (0.12.0-DEV), the representation V[3] corresponds to the second irreducible character above. 
This might change in future versions.)

\subsubsection*{Over $\Q$ with \textsc{Oscar}}

Alternatively, one can immediately compute rational matrices, but the entries are usually not integral. 
Those are
\[
	\mathbf{D}^{(2)}(s_1)
=	\begin{pmatrix}
	-7/10 & -1/5 \\ -51/20 & 7/10
	\end{pmatrix}
	\tbox{and}
	\mathbf{D}^{(2)}(s_2)
=	\begin{pmatrix}
	-7/15 & 8/15 \\ 22/15 & 7/15
	\end{pmatrix} 
\]
and obtained as follows.
\begin{verbatim}
julia> using Oscar;
julia> W = symmetric_group(3);
julia> V = Oscar.irreducible_modules(AnticNumberField, W, minimal_degree = true);
julia> g = map(z -> map_entries(QQ, z), matrix.(action(V[3])));
julia> s1, s2 = g[2], g[1]*g[2]
\end{verbatim}

\subsubsection{Block diagonalization}

We now compute the symmetry adapted basis. 
The basis itself of course depends on the used method to obtain $\mathbf{D}^{(2)}(s_1)$ and $\mathbf{D}^{(2)}(s_2)$, but since the SDP is invariant under change of basis, the equivariant matrices always have the same block diagonal structure. 
Following the steps preceding \Cref{prop_ChangeOfBasis}, we can obtain a symmetry adapted basis for the isotypic decomposition $\C^{\Weights_1}$. 
A basis for the space of trigonometric polynomials up to degree $1$ according to this decomposition is given by 
\[
	\underbrace{
		\mathfrak{e}^{ - \fweight{1}}
		- \mathfrak{e}^{\fweight{1} - \fweight{2}}
		, \hspace{.2cm}
		\mathfrak{e}^{\fweight{1} - \fweight{2}}
		- \mathfrak{e}^{\fweight{2}}
		, \quad
		\mathfrak{e}^{\fweight{2} - \fweight{1}}
		- \mathfrak{e}^{\fweight{1}}
		, \hspace{.2cm}
		\mathfrak{e}^{ - \fweight{2}}
		- \mathfrak{e}^{\fweight{2} - \fweight{1}}
	}_{V^{(2)}_1 \times V^{(2)}_2}
	, \hspace{.4cm}
	\underbrace{
		1
		, \hspace{.2cm}
		\mathfrak{e}^{ - \fweight{1}}
		+ \mathfrak{e}^{\fweight{1} - \fweight{2}}
		+ \mathfrak{e}^{\fweight{2}}
		, \hspace{.2cm}
		\mathfrak{e}^{ - \fweight{2}}
		+ \mathfrak{e}^{\fweight{2} - \fweight{1}}
		+ \mathfrak{e}^{\fweight{1}}
	}_{V^{(3)}_1 \times V^{(3)}_2 \times V^{(3)}_3}
	.
\]
It is the output of the following \textsc{Maple} command. 
\begin{verbatim}
maple> with(LinearAlgebra): 
maple> read("GeneralizedChebyshev.mpl"): with(GeneralizedChebyshev):
maple> IsotypicDecompositionBasisMultiplicative(A,2,1);
\end{verbatim}
By putting the basis elements in the correct order, we obtain a matrix $\tilde{\mathbf{T}}$, so that, for $s\in\weyl$, the matrices $\tilde{\mathbf{T}}^\dagger \, \vartheta(s) \, \tilde{\mathbf{T}}$ have the block diagonal structure 
\[
	\begin{pmatrix}
			\begin{bmatrix} \textcolor{red}{\blacksquare} & \textcolor{red}{\blacksquare} \\ \textcolor{red}{\blacksquare} & \textcolor{red}{\blacksquare} \end{bmatrix} & & & &	\\
	& 		\begin{bmatrix} \textcolor{red}{\blacksquare} & \textcolor{red}{\blacksquare} \\ \textcolor{red}{\blacksquare} & \textcolor{red}{\blacksquare} \end{bmatrix} & & & 		\\
	& & 	\begin{bmatrix} \textcolor{blue}{\blacksquare} \end{bmatrix} & & \\
 	& & & 	\begin{bmatrix} \textcolor{blue}{\blacksquare} \end{bmatrix} & \\
	& & & &	\begin{bmatrix} \textcolor{blue}{\blacksquare} \end{bmatrix} \\
	 
	\end{pmatrix},
\]
that is, $m_2=2$ blocks of size $d_2=2$ and $m_3=3$ blocks of size $d_3=1$. 

By reordering the columns, we obtain a matrix $\mathbf{T}$, so that, for $\mathbf{X} \in \mathrm{Toep}_d^\weyl$, the matrices ${\mathbf{T}}^\dagger \, \mathbf{X} \, {\mathbf{T}}$ have the block diagonal structure
\[
\begin{pmatrix}
	\begin{bmatrix} \textcolor{red}{\blacksquare} & \textcolor{red}{\blacksquare} \\ \textcolor{red}{\blacksquare} & \textcolor{red}{\blacksquare} \end{bmatrix} & & \\
	& \begin{bmatrix} \textcolor{red}{\blacksquare} & \textcolor{red}{\blacksquare} \\ \textcolor{red}{\blacksquare} & \textcolor{red}{\blacksquare} \end{bmatrix} & \\
	& & \begin{bmatrix} \textcolor{blue}{\blacksquare} & \textcolor{blue}{\blacksquare} & \textcolor{blue}{\blacksquare} \\
	\textcolor{blue}{\blacksquare} & \textcolor{blue}{\blacksquare} & \textcolor{blue}{\blacksquare} \\
	\textcolor{blue}{\blacksquare} & \textcolor{blue}{\blacksquare} & \textcolor{blue}{\blacksquare} \end{bmatrix} \\
\end{pmatrix},
\]
that is, $d_2=2$ equal blocks of size $m_2=2$ and $d_3=1$ equal block of size $m_3=3$. 

Specifically, a matrix $\mathbf{X} \in \mathrm{Toep}_1^\weyl$ is of the form
\[
\mathbf{X}
=	\begin{pmatrix}
	a & b & b & b & c & c & c \\
	\overline{b} & d & e & e & \mathrm{f} & \mathrm{f} & g \\
	\overline{b} & e & d & e & \mathrm{f} & g & \mathrm{f} \\
	\overline{b} & e & e & d & g & \mathrm{f} & \mathrm{f} \\
	\overline{c} & \overline{\mathrm{f}} & \overline{\mathrm{f}} & \overline{g} & h & k & k \\
	\overline{c} & \overline{\mathrm{f}} & \overline{g} & \overline{\mathrm{f}} & k & h & k \\
	\overline{c} & \overline{g} & \overline{\mathrm{f}} & \overline{\mathrm{f}} & k & k & h \\
	\end{pmatrix} 
	\phantom{-} \underrightarrow{\phantom{-} \mathbf{T}^\dagger \, \ldots \, \mathbf{T} \phantom{-}} \phantom{-}
	\begin{pmatrix}
	\textcolor{red}{d-e} & \textcolor{red}{\mathrm{f}-g} & 0 & 0 & 0 & 0 & 0 \\
	\textcolor{red}{\overline{\mathrm{f}}-\overline{g}} & \textcolor{red}{h-k} & 0 & 0 & 0 & 0 & 0 \\
	0 & 0 & \textcolor{red}{d-e} & \textcolor{red}{\mathrm{f}-g} & 0 & 0 & 0 \\
	0 & 0 & \textcolor{red}{\overline{\mathrm{f}}-\overline{g}} & \textcolor{red}{h-k} & 0 & 0 & 0 \\
	0 & 0 & 0 & 0 & \textcolor{blue}{a} & \textcolor{blue}{\sqrt{3}\,b} & \textcolor{blue}{\sqrt{3}\,c} \\
	0 & 0 & 0 & 0 & \textcolor{blue}{\sqrt{3}\,\overline{b}} & \textcolor{blue}{d+2\,e} & \textcolor{blue}{2\,\mathrm{f}+g} \\
	0 & 0 & 0 & 0 & \textcolor{blue}{\sqrt{3}\,\overline{c}} & \textcolor{blue}{2\,\overline{\mathrm{f}}+\overline{g}} & \textcolor{blue}{h+2\,k} \\
	\end{pmatrix} .
\]
for some $a,d,e,h,k\in\R,\,b,c,\mathrm{f},g\in\C$. 
Consider the $\weyl$-invariant trigonometric polynomial 
\[
	f
:=	6
+	4\,\mathfrak{e}^{  \fweight{1}}
+	4\,\mathfrak{e}^{  \fweight{2}}
+	4\,\mathfrak{e}^{  \fweight{2} -  \fweight{1}}
+	4\,\mathfrak{e}^{  \fweight{1} -  \fweight{2}}
+	4\,\mathfrak{e}^{- \fweight{1}}
+	4\,\mathfrak{e}^{- \fweight{2}}
+	2\,\mathfrak{e}^{ 2\fweight{1}}
+	2\,\mathfrak{e}^{ 2\fweight{2}}
+	2\,\mathfrak{e}^{ 2\fweight{2} - 2\fweight{1}}
+	2\,\mathfrak{e}^{ 2\fweight{1} - 2\fweight{2}}
+	2\,\mathfrak{e}^{-2\fweight{1}}
+	2\,\mathfrak{e}^{-2\fweight{2}}
\]
with $a = d = 6,\,b = c = \mathrm{f} = 7,\,e = k = 0, g = 14$. 
Then we have
\begin{align*}
		f^*
\geq 	f_{1,\weyl}^{\mathrm{block}}
=&		\min\limits_{\mathbf{X} \in \mathrm{Toep}_1^\weyl} \, \sum\limits_{i=1}^h d_i\,\trace(\tilde{\mathbf{F}}_i \, \tilde{\mathbf{X}}_i) 
		\quad 	\mathrm{s.t.} \quad 
		\trace(\mathbf{X}) = 1 ,\, \mathbf{T}^\dagger\,\mathbf{X}\,\mathbf{T} \mbox{\emph{ has blocks }} \tilde{\mathbf{X}}_i \succeq 0. \\
=&		\min\limits_{a,\ldots,k} \, 2\,(6\,d - 6\,e - 7\,f + 7\,g - 7\,\overline{f} + 7\,\overline{g} + 6\,h - 6\,k)\\
 &		\quad + (6\,a + 21\,b + 21\,c + 21\,\overline{b} + 6\,d + 12\,e + 56\,f + 28\,g + 21\,\overline{c} + 56\,\overline{f} + 28\,\overline{g} + 6\,h + 12\,k)\\
 &		\mathrm{s.t.} \quad 
		a+3\,d+3\,h = 1,\,
		\textcolor{red}{
		\begin{pmatrix}
			d-e & \mathrm{f}-g \\ \overline{\mathrm{f}}-\overline{g} & h-k
		\end{pmatrix}}
		\succeq 0,\,
		\textcolor{blue}{
		\begin{pmatrix}
			a & \sqrt{3}\,b & \sqrt{3}\,c \\ \sqrt{3}\,\overline{b} & d+2\,e & 2\,\mathrm{f}+g \\\sqrt{3}\,\overline{c} & 2\,\overline{\mathrm{f}}+\overline{g} & h+2\,k
		\end{pmatrix}}
		\succeq 0 .
\end{align*}
Now, we only need to certify positive semi-definiteness for Hermitian matrices with $2^2 + 3^2 = 13$ real entries instead of matrices with $7^2 = 49$ entries. 

\section*{Conclusion}

For large $n$, $d$ and $\nops{\weyl}$, solving the SDP via \Cref{thm_SDPBlock} is vastly more efficient than without the symmetry reduction via \Cref{eq_MinRelaxation}, because 
(i) the group symmetry reduces the number of variables and 
(ii) the matrices are blockdiagonal and thus the matrix size is reduced from $\nops{\Weights_d}^2$ to $m_1^2,\,\ldots,\,m_h^2$ (which simplifies the certification of positive semi-definiteness). 

The bottleneck in the presented approach is the computation of the representing matrices $\mathbf{D}^{(i)}_{k\,\ell}(s)$ for the irreducible representations. 
In the examples, we have used computer algebra systems, such as \textsc{Maple} and \textsc{Oscar}. 
Current work in progress revolves around the implementation of the presented approach to the point where the input consists of an invariant trigonometric polynomial $f$, a finite group $\weyl$ and an order of the relaxation $d$ and the output of the matrices defining the SDP from \Cref{thm_SDPBlock}. 

We compare with the approach via generalized Chebyshev polynomials in \cite{chromatic22,chromaticissac22}. 
The SDP matrices have blocks and the sizes of the blocks are
\begin{enumerate}
	\item dense: $\nops{\Weights_d}^2$ \hspace{.2cm} ($1$ Hermitian block), \phantom{$\sum\limits_{i=1}^h$}
	\item Chebyshev: $\frac{\nops{\Weights_d}^2 + n^2\, \nops{\Weights_{d-n}}^2}{\nops{\weyl}^2}$ \hspace{.2cm} ($2$ symmetric blocks), 
	\item symmetry adapted basis: $\sum\limits_{i=1}^h d_i\,\left(m_i^{(d)}\right)^2$ \hspace{.2cm} ($h$ Hermitian blocks with $d_i$ equal subblocks). 
\end{enumerate}
The dense approach is in any case the least efficient. 
To decide, whether $2.$ or $3.$ provides smaller matrices, one can compute $\nops{\Weights_d}$, $\nops{\weyl}$, $d_i$ and $m_i^{(d)}$ using the \textsc{Maple} package \textsc{GeneralizedChebyshev}: 
\begin{verbatim}
maple> with(LinearAlgebra): read("GeneralizedChebyshev.mpl"): with(GeneralizedChebyshev): 
maple> Type,n,d := A,2,1; 
maple> nops(WeightList(Type,n,d));              # number of lattice points up to degree d
maple> WeylGroupOrder(Type,n);                                  # order of the Weyl group
maple> IrredRepDimMultiplicative(Type,n);     # dimensions of irreducible representations
maple> IsotypicDecompositionMultiplicitiesMultiplicative(Type,n,d);      # multiplicities
\end{verbatim}

\section*{Open problems}

How to decide a priori which approach gives a better numerical result is unkown to the author. 
By \Cref{thm_SDPBlock}, we already know that the approach via symmetry adapted basis gives the same solution as the one for the dense approach. 
The exponential convergence rate is proven in \cite{bach22}. 
The Chebyshev approach on the other hand uses polynomial instead of trigonometric optimization techniques. 
In this case, the convergence rate is given in \cite{baldimourrain23}. 
This article now provides the means to compare the two approaches, which the author intends to address in future work. 

For the symmetric group $\mathfrak{S}_3$ in \Cref{table_A2_multiplicities}, 
one can observe a pattern for the multiplicities, namely 
$m_1^{(d+1)} = m_1^{(d)} +   d    $, 
$m_2^{(d+1)} = m_2^{(d)} + 2 d + 2$ and
$m_3^{(d+1)} = m_3^{(d)} +   d + 2$. 
It would be very interesting to prove more general formulae for the behavior of the multiplicities. 

\section*{Acknowledgments}

I wish to thank Evelyne Hubert (Sophia Antipolis), Cordian Riener (Troms\o) and Ulrich Thiel (Kaiserslautern) 
for their lessons and fruitful discussions on representation theory. 

This work has been supported by the Deutsche Forschungsgemeinschaft transregional collaborative research centre (SFB--TRR) 195 ``Symbolic Tools in Mathematics and their Application''. 

\pagestyle{fancy}
\lhead[T. Metzlaff]{}
\rhead[]{}
\cfoot{}
\rfoot[\today]{\thepage}
\lfoot[\thepage]{}
\renewcommand{\headsep}{8mm}
\renewcommand{\footskip}{15mm}

\bibliographystyle{alpha}
{\bibliography{mybib.bib}}

\clearpage

\appendix

\section{Lattices and Weyl groups associated to crystallographic root systems}
\label{appendix_root_systems}

We recall the definition for the lattices and groups, which appeared in the examples throughout the article and which are of interest for the applications mentioned in the introduction. 
Such lattices arise from root systems, which appear, for example, in the representation theory of semi-simple complex Lie algebras, see \cite{bourbaki456,serre13}. 

Let $V$ be a finite-dimensional real vector space with inner product $\sprod{\cdot,\cdot}$. 
A subset $\Roots\subseteq V$ is called a (crystallographic, reduced) \textbf{root system} in $V$, if the following conditions\footnote{This is the definition given in \cite[Ch. VI, \S 1, D\'{e}f. 1]{bourbaki456} plus the ``reduced'' property R4.} hold.
\begin{enumerate}
	\item[R1] $\Roots$ is finite, spans $V$ as a vector space and does not contain $0$.
	\item[R2] If $\rho, \tilde{\rho} \in \Roots$, then $\langle\tilde{\rho},\rho^\vee\rangle \in \Z$, where $\rho^\vee:=\frac{2\,\rho}{\langle\rho,\rho\rangle}$.
	\item[R3] If $\rho, \tilde{\rho} \in \Roots$, then $s_\rho(\tilde{\rho}) \in \Roots$, where $s_\rho$ is the  reflection defined by $ s_\rho(u) = u - \langle u,\rho^\vee\rangle \rho$ for $u \in  V$.
	\item[R4] For $\rho \in \Roots$ and $c \in \R$, we have $c\rho \in \Roots$ if and only if $c = \pm 1$.$\phantom{\frac{1}{2}}$
\end{enumerate}
The elements of $\Roots$ are called \textbf{roots} and the $\rho^\vee$ are called \textbf{coroots}. 
The lattice $\Corootlattice$ spanned by all coroots $\roots^\vee$ is full-dimensional in $V$ and called the \textbf{coroot lattice}. 

The \textbf{Weyl group} $\weyl$ of $\Roots$ is the group generated by the reflections $s_\roots$ for $\roots \in \Roots$. 
This is a finite subgroup of $\mathrm{GL}(V)$ and orthogonal with respect to the inner product $\sprod{\cdot , \cdot }$. 
The coroot lattice is a $\weyl$-lattice and the group product of $\weyl$ by $\Corootlattice$ is semi-direct. 

Every root system contains a \textbf{base}, 
that is, a subset $\Base=\{\rho_1,\ldots,\rho_n\}$ of $\Roots$ satisfying the following conditions \cite[Ch. VI, \S 1, Thm. 3]{bourbaki456}.
\begin{enumerate}
	\item[B1] $\Base$ is a vector space basis of $V$.
	\item[B2] Every root $\rho \in \Roots$ can be written as $\rho = \alpha_1 \, \rho_1 + \ldots + \alpha_n \, \rho_n$ or $\rho = - \alpha_1 \, \rho_1 - \ldots - \alpha_n \, \rho_n$ for some $\alpha \in \N^n$.
\end{enumerate}

A \textbf{weight} of $\Roots$ is an element $\weight\in V$, such that, for all $\roots\in\Roots$, we have $\sprod{\weight,\roots^\vee}  \in \Z$. 
By the ``crystallographic'' property R2, every root is a weight. 
For a base $\Base=\{\rho_1,\ldots,\rho_n\}$ of the root system, the \textbf{fundamental weights} are the elements $\{ \fweight{1}, \ldots , \fweight{n}\}$, where, for $1\leq i,j \leq n$, we have $\sprod{ \fweight{i}, \roots_j^\vee} = \delta_{i,j}$. 
The set $\Weights$ of all weights $\weight$ is a full-dimensional $\weyl$-lattice in $V$ and called the \textbf{weight lattice} of $\Roots$. 
By definition, it is the dual lattice of the coroot lattice, that is, $\Weights^* = \Corootlattice$. 

Assume that $V=V^{(1)}\oplus\ldots\oplus V^{(k)}$ is the direct sum of proper orthogonal subspaces and that, for each $1\leq i\leq k$, there is a root system $\Roots^{(i)}$ in $V^{(i)}$. 
Then $\Roots:=\Roots^{(1)}\cup\ldots\cup\Roots^{(k)}$ is a root system in $V$ called the \textbf{direct sum} of the $\Roots^{(i)}$. 
If a root system is not the direct sum of at least two root systems, then it is called \textbf{irreducible}, see \cite[Ch. VI, \S 1.2]{bourbaki456}. 

Every root system can be uniquely decomposed into irreducible components \cite[Ch. VI, \S 1, Prop. 6]{bourbaki456} 
and there are only finitely many cases \cite[Ch. VI, \S 4, Thm. 3]{bourbaki456} 
denoted by $\RootA$, $\RootB$, $\RootC$ $(n\geq 2)$, $\RootD$ $(n\geq 4)$, $\RootE[6,7,8]$, $\RootF[4]$ and $\RootG[2]$. 
These are explicitly given in \cite[Pl. I - IX]{bourbaki456}. 

The Weyl group $\weyl$ is the product of the Weyl groups corresponding to the irreducible components, see the discussion before \cite[Ch. VI, \S 1, Prop. 5]{bourbaki456}. 
If $\Roots$ is irreducible and $\Base$ is a fixed base, 
then there exists a unique positive root $\highestroot \in \Roots^+$, 
so that, for all $\roots\in\Roots$, there is some $\alpha\in\N^n$ with $\highestroot - \roots = \alpha_1\,\roots_1 + \ldots + \alpha_n\,\roots_n$ \cite[Ch. VI, \S 1, Prop. 25]{bourbaki456}. 
We call $\highestroot$ the \textbf{highest root}. 

\begin{example}\label{example_2RootSys}
	For $n=1$, a root system in $\R$ must be of the form $\Roots=\{ \pm\roots \}$ for some $\roots\in\R_{>0}$, which is the only base element and the highest root. 
	It admits the reflection at the origin $s_\roots = -1$ and so the Weyl group is $\weyl = \{ \pm 1 \}$. 
	The fundamental Weyl chamber is $\PC = \R_{>0}$ and the coroot is $\roots^\vee = 2\,\roots / \sprod{\roots , \roots} = 2 / \roots$. 
	For $\fweight{}\in\R$ to be the fundamental weight, we require $1 = \sprod{\fweight{} , \roots^\vee} = 2\,\fweight{} / \roots$, that is, $\fweight{} = \roots / 2$. 
	When we choose $\roots=2$, then $\Roots$ admits the self-dual lattice $\Weights = \Corootlattice = \Z$. 
	
	For $n=2$, the irreducible root systems are depicted in \emph{\Cref{example_rootsystemA2,example_rootsystemB2,example_rootsystemC2,example_rootsystemG2}}. 
	The roots are depicted in green, the base in red and the fundamental weights in blue. 
	The gray shaded region is a Vorono\"{i} cell of the coroot lattice $\Corootlattice:$ 
	we have two squares $(\RootC[2]$ and $\RootB[2])$ and two hexagons $(\RootA[2]$ and $\RootG[2])$. 
	The blue shaded triangle is a fundamental domain of the semi-direct product $\weyl\ltimes\Corootlattice$. 
	
	\begin{minipage}{0.45\textwidth}
		\begin{figure}[H]
			\begin{minipage}{0.4\textwidth}
				\begin{flushright}
					\begin{overpic}[width=\textwidth,grid=false,tics=10]{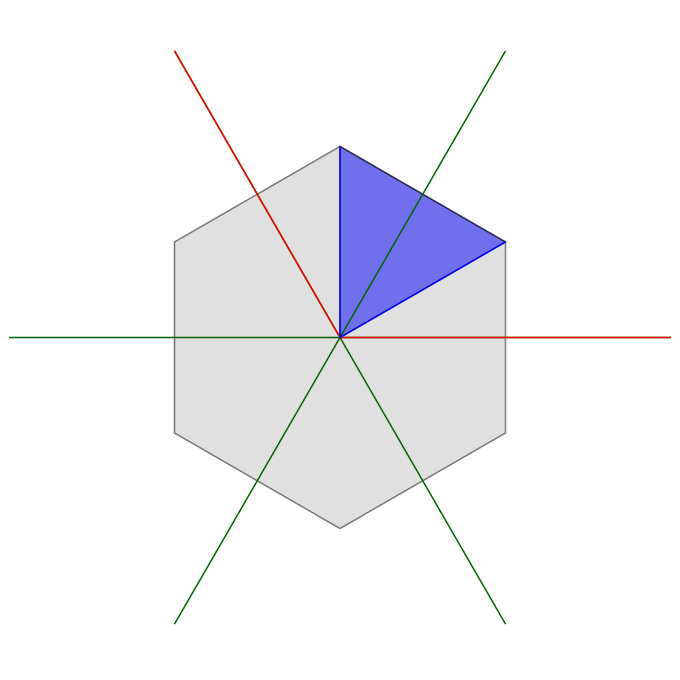}
						\put (20,95) {\large $\displaystyle \rho_2$}
						\put (95,55) {\large $\displaystyle \rho_1$}
						\put (75,66) {\large $\displaystyle \fweight{1}$}
						\put (50,81) {\large $\displaystyle \fweight{2}$}
					\end{overpic}
				\end{flushright}
			\end{minipage} \hfill
			\begin{minipage}{0.5\textwidth}
				$\weyl(\RootA[2]) \cong \mathfrak{S}_3$\\
				$\fweight{1}=[ 2,-1,-1]^t /3$\\
				$\fweight{2}=[ 1, 1,-2]^t /3$\\
				$\rho_{1}=[1,-1,0]^t =\rho_{1}^\vee$\\
				$\rho_{2}=[0,1,-1]^t =\rho_{2}^\vee$\\
				$\highestroot = \rho_{1}^\vee + \rho_{2}^\vee$
			\end{minipage}
			\centering
			\caption{The root system $\RootA[2]$ in $\R^3/\langle [1,1,1]^t \rangle$.}\label{figA2}
			\label{example_rootsystemA2}
		\end{figure}
	\end{minipage}\hfill
	\begin{minipage}{0.45\textwidth}
		\begin{figure}[H]
			\begin{minipage}{0.4\textwidth}
				\begin{flushright}
					\begin{overpic}[width=\textwidth,,tics=10]{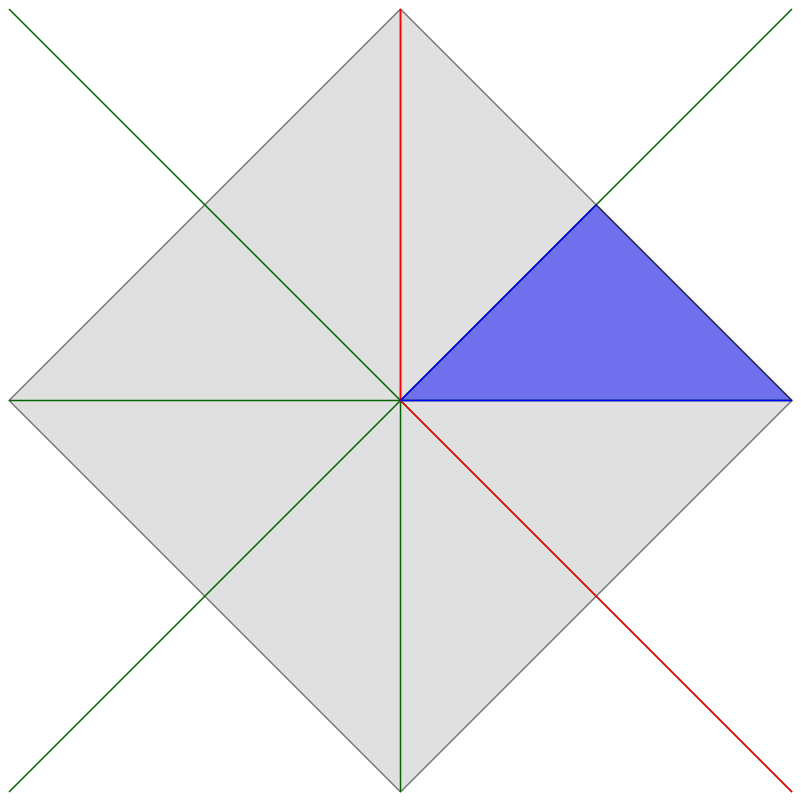}
						\put (53,100) {\large $\displaystyle \rho_2$}
						\put (99,  8) {\large $\displaystyle \rho_1$}
						\put (69, 83) {\large $\displaystyle \fweight{2}$}
						\put (95, 55) {\large $\displaystyle \fweight{1}$}
					\end{overpic}
				\end{flushright}
			\end{minipage} \hfill
			\begin{minipage}{0.5\textwidth}
				$\weyl(\RootB[2]) \cong \mathfrak{S}_2\ltimes\{\pm 1\}^2$\\
				$\fweight{1}=[ 1,0]^t$\\
				$\fweight{2}=[ 1,1]^t /2$\\
				$\rho_{1}=[1,-1]^t =\rho_{1}^\vee$\\
				$\rho_{2}=[0,1]^t =\rho_{2}^\vee/2$\\
				$\highestroot=\rho_{1}^\vee+\rho_{2}^\vee$
			\end{minipage}
			\centering
			\caption{The root system $\RootB[2]$ in $\R^2$.}\label{example_rootsystemB2}\label{figB2}
		\end{figure}
	\end{minipage}
	
	~\vspace{.2cm}
	
	\begin{minipage}{0.45\textwidth}
		\begin{figure}[H]
			\begin{minipage}{0.4\textwidth}
				\begin{flushright}
					\begin{overpic}[width=\textwidth,grid=false,tics=10]{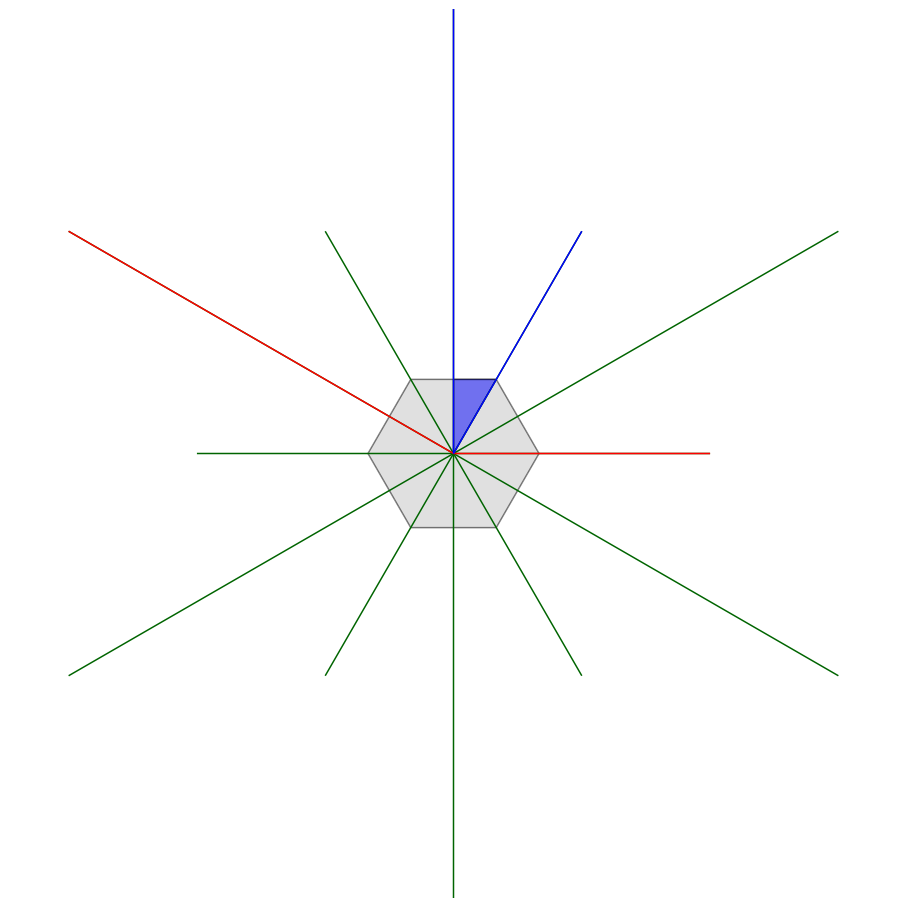}
						\put ( 15 , 75 ) {\large $\displaystyle \rho_2$}
						\put ( 80 , 55 ) {\large $\displaystyle \rho_1$}
						\put ( 52 , 95 ) {\large $\displaystyle \fweight{2}$}
						\put ( 65 , 75 ) {\large $\displaystyle \fweight{1}$}
					\end{overpic}
				\end{flushright}
			\end{minipage} \hfill
			\begin{minipage}{0.5\textwidth}
				$\weyl(\RootG[2])\cong\mathfrak{S}_3\ltimes \{\pm 1\}$\\
				$\fweight{1}=[ 0,-1, 1]^t$\\
				$\fweight{2}=[-1,-1, 2]^t$\\
				$\rho_{1}   =[ 1,-1, 0]^t = \rho_{1}^\vee$\\
				$\rho_{2}   =[-2, 1, 1]^t = 3\,\rho_{1}^\vee $\\
				$\highestroot = 3\,\rho_{1}^\vee + 6\,\rho_{2}^\vee$
			\end{minipage}
			\centering
			\caption{The root system $\RootG[2]$ in $\R^3/\langle [1,1,1]^t \rangle$.}\label{example_rootsystemG2}
		\end{figure}
	\end{minipage}\hfill
	\begin{minipage}{0.45\textwidth}
		\begin{figure}[H]
			\begin{minipage}{0.4\textwidth}
				\begin{flushright}
					\begin{overpic}[width=\textwidth,,tics=10]{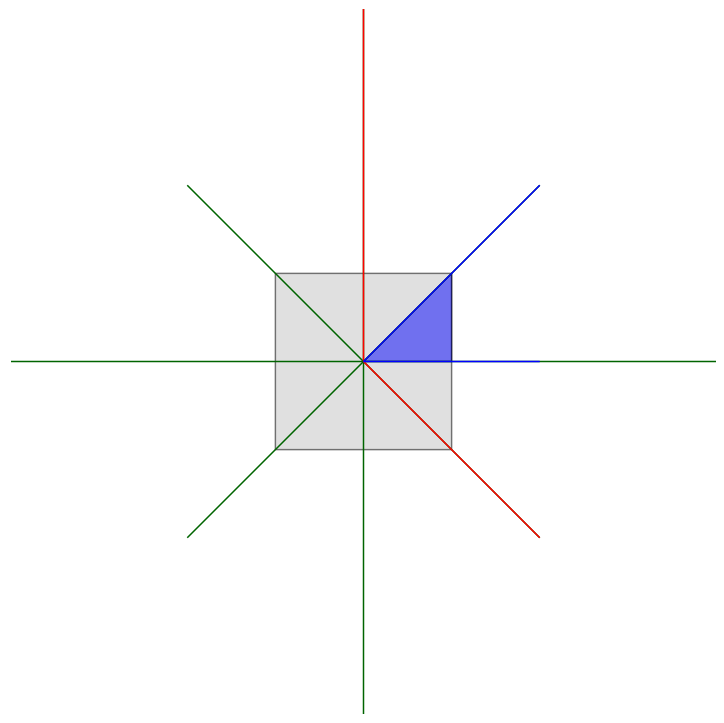}
						\put (53, 95) {\large $\displaystyle \rho_2$}
						\put (77, 20) {\large $\displaystyle \rho_1$}
						\put (69, 80) {\large $\displaystyle \fweight{2}$}
						\put (77, 54) {\large $\displaystyle \fweight{1}$}
					\end{overpic}
				\end{flushright}
			\end{minipage} \hfill
			\begin{minipage}{0.5\textwidth}
				$\weyl(\RootC[2]) \cong \mathfrak{S}_2\ltimes\{\pm 1\}^2$\\
				$\fweight{1}=[ 1,0]^t $\\
				$\fweight{2}=[ 1,1]^t $\\
				$\rho_{1}=[1,-1]^t =\rho_{1}^\vee$\\
				$\rho_{2}=[0,2]^t = 2\,\rho_{2}^\vee$\\
				$\highestroot = 2\,\rho_{1}^\vee + 2\,\rho_{2}^\vee$
			\end{minipage}
			\centering
			\caption{The root system $\RootC[2]$ in $\R^2$.}\label{example_rootsystemC2}\label{figC2}
		\end{figure}
	\end{minipage}
\end{example}

\end{document}